\documentclass[11pt]{amsart}
\usepackage{amssymb,amsmath,amsfonts,amsthm}
\usepackage{fullpage}

\newtheorem{theorem}{Theorem}

\newtheorem{corollary}[theorem]{Corollary}

\begin{document}

\title{A note on monotonicity of mixed Ramsey numbers}
\author{Maria Axenovich} 
\author{JiHyeok  Choi}
 
\thanks{Department of Mathematics, Iowa State University, Ames, IA 50011}
\thanks{The project was supported in part by   NSA grant H98230-09-1-0063 and NSF grant DMS-0901008}
\begin{abstract}   
For two graphs, $G$, and $H$, an edge-coloring of a complete graph is $(G,H)$-good 
if there is no monochromatic subgraph isomorphic to $G$ and no rainbow subgraph isomorphic to 
$H$ in this coloring. The set of number of colors used by some $(G,H)$-colorings of 
$K_n$ is called a mixed-Ramsey spectrum. This note addresses a fundamental question of 
whether the spectrum is an interval. It is shown that the answer is ``yes''
if $G$ is not a star and $H$ does not contain a pendent edge.
\end{abstract}

\maketitle

\section{Introduction}

Let $G$ and $H$ be two graphs on fixed number of
vertices. An edge coloring of a complete graph, $K_n$,  is called
$(G,H)$-good if there is no monochromatic copy of $G$ and no
rainbow (totally multicolored) copy of $H$ in this coloring. This, sometimes called mixed-Ramsey 
coloring, is a hybrid of classical Ramsey and anti-Ramsey colorings, \cite{R, ESS}.
As shown by Jamison and West \cite{JW}, a $(G,H)$-good coloring of an arbitrarily large complete graph exists unless either $G$ is a star or  $H$ is a forest.

Let $S(n; G, H)$ be the set of the number of colors, $k$, such that there is 
a $(G,H)$-good coloring of $K_n$ with $k$ colors. We call $S(n; G, H)$ a spectrum.
Let $\max S(n; G, H)$, $\min S(n; G, H)$  be the maximum, minimum  number in $S(n; G, H)$, respectively.
The behavior of these functions was studied in \cite{AI}, \cite{FM}, \cite{AC} and others.
Note that if there is no restriction on a graph $H$, $S(n; G, *)$ is an interval $[k, \binom{n}{2}]$, 
 where $k$ is the largest number such that 
 $r_{k-1}(G)\leq n$, a classical multicolor Ramsey number.

The main question investigated in this note is whether the same behavior continues to 
hold for mixed Ramsey colorings. Specifically,  for given integer $n$ and graphs $G$ and $H$, 
is $S(n; G, H)$ an interval?
When $G$ is not a star, for most graphs $H$, we show that $S(n; G, H)$ is an interval.

\begin{theorem}\label{nopendent} Let $G$ be  a graph that is not a star, and let $H$ be a graph with minimum degree at least $2$. 
Then for any natural number  $n$, $S(n; G,H)$ is an interval.
\end{theorem}

The simplest connected graph $H$ which is not a tree and which has  a vertex of degree $1$  is $K_3+e$, 
a $4$-vertex graph obtained by attaching a pendent edge to a triangle.
We show that $S(n; G, K_3+e)$  could have a gap for some graphs $G$ and 
some values of $n$. However,  when $n$ is arbitrarily large, 
we do not have a single example of a graph $G$ and a graph $H$ for 
which $S(n; G, H)$ is not an interval.

Specifically, the next theorem is a collection of results on $S(n; G,K_3+e)$.
Here, $\ell K_2$ is a matching of size $\ell$, $C_4$ is a $4$-cycle, and $P_4$ is a path on $4$ 
vertices.

\begin{theorem}\label{others}

\noindent
\begin{itemize}
\item{}
$S(n; \ell K_2, K_3)= S(n; \ell K_2, K_3+e) = [\lceil\frac {n-2\ell+1}{\ell-1}\rceil+1, n-1]$,  $n\geq 4 $,\\
$S(n; P_4, K_3)= S(n; P_4, K_3+e) = [n-2, n-1]$, $n\geq 4$,\\
$S(n; C_4, K_3)= S(n;  C_4, K_3+e) = [n-3, n-1]$,  $n\geq r_3(C_4)=11$,\\
$S(n; K_3, K_3) = S(n; K_3, K_3+e) = [ c \log n, n-1]$, $n\geq  r_3(K_3)=17$,\\
$S(n; K_{1, \ell}, K_3)= S(n; K_{1, \ell}, K_3+e)= \emptyset$, $n \geq  3\ell+1$.\\
\item{}
$S(10;C_4,K_3+e)=\{3,7,8,9\}$.\\
\end{itemize}
\end{theorem}

\begin{corollary}
If $\ell\geq 2$ and $n \geq \max \{17, 3\ell+1\}$, then $S(n; G, K_3+e)$ is an interval for any $G\in \{K_3, \ell K_2, C_4, P_4, K_{1, \ell}\}$. However,  $S(n; G, K_3+e)$ is not an interval if $n=10$ and $G=C_4$. 
\end{corollary}

\noindent
{\bf Open question.}
Are there graphs $G$ and $H$ such that for any natural number $N$ there is $n>N$
so that $S(n; G, H)$ is not an interval?

~\\~\\

\section {Definitions and proofs of main results}

For an edge coloring $c$ of $K_n$ and a vertex $x\in V(K_n)$,
let $N_c(x)$ be the set of colors used only on edges incident to $x$, and for $X\subseteq V(K_n)$ let $c(X)$ be the set of colors used on edges induced by $X$.
Let $|c|$ denote the number of colors used in the coloring $c$.
Then $|c| = |N_c(x)| + |c(V\setminus x)|$ for any $x\in V$.
We shall use function 
$$f(k;G,H)=\max  \{n : \mbox{ there is a $(G, H)$-good coloring of $K_n$ using exactly $k$ colors}\}.$$
Note that if $f(k;G,H) = n$, then $\min S(n;G,H)= k$.\\~\\

\noindent
{\bf Observation $1$ } 
If  $G$ is not a star, and $A$ and $B$ are color classes which are stars with the same center 
in a $(G,H)$-good coloring $c$ of $K_n$ with $k$ colors,  then 
replacing $A$ and $B$ in $c$  with a new color class $A\cup B$ gives a $(G, H)$-good coloring
using $k-1$ colors.\\

\noindent
{\bf Observation $2$ } 
For any graphs $G$ and $H$,
$$\min S(n;G,H) \leq \min S(n+1,G,H).$$
\begin{proof}
Consider a $(G,H)$-good coloring of $K_{n+1}$ with $k$ colors. Delete one vertex to 
get a $(G,H)$-good coloring of $K_{n}$ with $ k'\leq k$ colors. 
\end{proof}

\noindent
{\bf Observation $3$ } 
For $G\subseteq G'$ and $H\subseteq H'$,
$$S(n;G,H)\subseteq S(n;G',H)\subseteq S(n;G',H')\quad\mbox{ and }\quad S(n;G,H) \subseteq S(n;G,H') \subseteq S(n;G',H').$$

\begin{proof}
If there is no monochromatic $G$ and no rainbow $H$ in a coloring of $E(K_n)$, then there is no monochromatic $G'$ and no rainbow $H'$ in this coloring.
\end{proof}


\noindent
{\bf Observation $4$ } 
If $G$ is not a star, $H$ has minimum degree at least $2$, and $k\in S(n;G,H)$, 
then $k+1\in S(n+1;G,H)$.
\begin{proof}
Consider a $(G,H)$-good coloring of $K_{n}$ with $k$ colors. Add a new vertex $x$, and color edges incident to x by a new color to get a $(G,H)$-good coloring of $K_{n+1}$ with $ k+1$ colors.
\end{proof}
~\\

\begin{proof}[Proof of Theorem \ref{nopendent}]
~\\
We need to prove that $[\min S(n;G,H), \max S(n;G,H)]\subseteq S(n;G,H)$. 
We use induction on $n$. 
When $n=2$, any coloring uses one color. 
Let $n\geq 3$. 
Consider the smallest $k$ such that $[k,\max S(n;G,H)]\subseteq S(n;G,H)$. Observe that in any $(G,H)$-good $k$-coloring of $K_n$ and any vertex $x$, we have $|N(x)|\leq 1$, otherwise applying 
Observation $1$ gives us a $(G,H)$-good $(k-1)$-coloring of $K_n$  violating minimality of $k$. 
Consider a $(G,H)$-good $k$-coloring of $K_n$ and any vertex $x$, and delete it. Then we have a $(G,H)$-good coloring of $K_{n-1}$ with $k$ or $k-1$ colors. Here we note that 
$\max S(n-1;G,H)\geq k-1$. By induction, $S(n-1;G,H)$ is an interval, i.e., $[\min S(n-1; G, H), \max S(n-1; G, H)] = S(n-1; G,H)$. Then by Observation $4$, 
$[\min S(n-1;G,H)+1, \max S(n-1;G,H)+1] \subseteq S(n;G,H)$. 
Since $\min S(n;G,H)\geq \min S(n-1;G,H)$ from Observation 2, 
$[\min S(n;G,H), \max S(n-1;G,H)+1] \subseteq S(n;G,H)$. 
Since $k\leq \max S(n-1;G,H) +1$ and $[k,\max S(n;G,H)]\subseteq S(n;G,H)$  we finally have that $[\min S(n;G,H), \max S(n;G,H)]\subseteq S(n;G,H)$.
\end{proof}
~\\

\begin{proof}[Proof of Theorem \ref{others}]
~\\
First observe that 
$\max S(n;G,H)\leq AR(n,H)$, where $AR(n, H)$ is the classical anti-Ramsey 
number, the maximum number of colors in an edge-coloring of $K_n$ with no rainbow 
subgraphs isomorphic to $H$. 
If  $G$ is not a star, $\max S(n;G,K_3)= AR(n,K_3)=n-1$, see \cite{AI}.
Moreover, from Observation 3, we obtain that 
$\max S(n;G,K_3)\leq \max S(n;G,K_3+e)$; and from \cite{Gor},  we know that 
$AR(n,K_3)=AR(n,K_3+e)$. Thus, when $G$ is not a star, $\max S(n;G,K_3)=\max S(n;G,K_3+e)=n-1$ for $n\geq 4$.

Therefore if $\min S(n; G, K_3) = \min S (n, G, K_3+e)$, and $G$ is not a star, we
can conclude that $S(n; G, K_3+e) = S(n; G, K_3)$, which is an interval by Theorem $1$.
Next, we shall analyze $\min S (n, G, K_3+e)$.
 Recall that  $\min S(n;G,H)= k$ if $f(k,G,H)= n$.
Moreover, $f(k,G,H)+1\leq r_k(G)$, where  $r_k(G)$ denotes the classical $k$-color Ramsey number for $G$. 
The equality holds if there is a $k$-coloring of $E(K_{r_k(G)-1})$ with 
no monochromatic $G$ and no rainbow $H$.\\

Case 1.  $G=\ell K_2$\\
From  \cite{Rad}, we have that $r_k(\ell K_2) = (k-1)(\ell-1)+2\ell$.
 The extremal coloring providing this Ramsey number can be constructed as follows.
Consider a complete graph on $2\ell-1$ vertices colored entirely with color 1, 
add $\ell-1$ vertices and color all edges incident to these vertices with color 2, 
then add another $\ell-1$ vertices and color all edges incident to these vertices with color 3. 
Repeat this process until we get a $k$-coloring of a complete graph on $2\ell-1 + (k-1)(\ell-1)$ vertices which contains no monochromatic $\ell K_2$. 
Note that this coloring contains no rainbow cycles, thus, it contains neither rainbow copy of $K_3$ nor rainbow copy of $K_3+e$.
Hence $\min S(n;\ell K_2, H) = \min S(n;\ell K_2,H+e)$ for any $H$, not a forest.
In particular for $\ell\geq2$, $\min S(n;\ell K_2,K_3) = \min S(n;\ell K_2,K_3+e) = \lceil\frac {n-2\ell+1}{\ell-1}\rceil+1$.

Case 2.  $G \in \{K_3, P_4, C_4\}$\\
From   \cite{CG,AI,GSSS, FGJM, FM} we have  that
$f(k,K_3,K_3)=f(k,K_3,K_3+e)= \lambda(k)$, for $k\geq 4$,
where $\lambda(k)=5^{k/2}$ if $k$ is even, $2\cdot 5^{(k-1)/2}$ if $k$ is odd; 
$f(k,P_4,K_3)=f(k,P_4,K_3+e)= k+2 \,\mbox{ for } k\geq 1$, and 
$f(k,C_4,K_3)=f(k,C_4,K_3+e)= k+3 \,\mbox{ for } k\geq 4$.
Therefore $\min S(n;P_4,K_3)=\min S(n;P_4,K_3+e)=n-2$, 
$\min S(n;C_4,K_3)=\min S(n;C_4,K_3+e)=n-3$, and $\min S(n;K_3,K_3)=\min S(n;K_3,K_3+e)=c\log n$.
Thus $\min S(n; G, K_3)= \min S(n; G, K_3+e)$ for $G\in \{K_3, P_4, C_4\}$ and $n\geq r_3(G)$.

Case 3.   $G = K_{1, \ell}$\\
In \cite{GS}, it was shown that any coloring of $E(K_n)$ with no rainbow triangles  has a monochromatic 
star $K_{1,2n/5}$.  Using this fact and the pigeonhole principle, we easily see that any coloring of $E(K_n)$ with no rainbow $K_3+e$ has a monochromatic star $K_{1,n/3}$. This is sharp as is seen in \cite{FM}.
Therefore $S(n;K_{1,\ell},K_3)=S(n; K_{1, \ell}, K_3+e) = \emptyset$ if $n>3\ell$.\\

Summarizing 1), 2),  and 3) we have that $S(n; G, K_3) = S(n; G, K_3+e)$ is an interval if 
 $G$ is one of  $\{\ell K_2$,  $K_3$,  $P_4$,  $C_4$,   $K_{1, \ell}\}$ and  $n\geq N$, where $N$ is a constant depending only on $G$.
 This concludes the proof of the first part of the Theorem.\\

Consider the case when $G=C_4$, $H=K_3+e$  and $n=10$.
Since $r_2(C_4)=6<10$, we see that there is  no $(C_4, K_3+e)$-good coloring of $K_{10}$ in two colors.
On the other hand, since $r_3(C_4)=11$, there is a $(C_4, K_3+e)$-good coloring of $K_{10}$ in three colors. Thus
$\min S(10; C_4, K_3+e)=3$. 
We also have that $ \max S(10; C_4, K_3+e)=AR(10,K_3)=9$.
Since $f(k,C_4,K_3+e)= k+3<10$ for $4\leq k\leq 6$, there is no $(C_4,K_3+e)$-good coloring of $K_{10}$ with $4$, $5$, or $6$ colors.
To construct $8$- and $7$-colorings of $K_{10}$ with no rainbow $K_3+e$ and no monochromatic $C_4$, 
consider a vertex set  $\{v_1, \ldots, v_{10}\}$. 
Let $c(v_iv_j)= i$, $1\leq i\leq 7$, $i<j$;  $c(v_8v_9)=c(v_8v_{10})=c(v_9v_{10})= 8$.
Let $c'(v_iv_j)= i$, $1\leq i\leq 5$, $i<j$;  $c'(v_6v_7)=c'(v_7v_8)=c'(v_8v_9)=c'(v_9v_{10})=c'(v_{10}v_6)=6$, all other edges get color $7$ under $c'$.
Note that $c$ and $c'$ are $8$- and $7$-colorings, respectively,  containing no rainbow $K_3$ and no monochromatic $C_4$.
Thus $S(10;C_4,K_3+e)=\{3,7,8,9\}$.

 \end{proof}

\end{document}